\documentclass{amsart}
\usepackage{amsmath,amssymb,amsthm}
\usepackage{tikz, amscd, verbatim, graphicx, xcolor, color, extarrows}
\usepackage{chapterbib}
\usepackage{url}
\usetikzlibrary{shapes}
\newcommand{\C}{\mathbb{C}}
\renewcommand{\P}{\mathbb{P}}
\newcommand{\gG}{\Gamma}
\newcommand{\gr}{\gamma}
\newcommand{\bo}{\mathcal{O}}
\newcommand{\ol}[1]{\overline{#1}}

\newcommand{\iso}{\cong}
\renewcommand{\l}{\ell}
\newcommand{\gx}{\chi}
\newcommand{\oo}{\infty}
\newcommand{\set}[1]{\left\{ #1 \right\}}

\numberwithin{equation}{section}

\theoremstyle{plain}
\newtheorem{theorem}{Theorem}[section]
\newtheorem{prop}[theorem]{Proposition}
\newtheorem{coro}[theorem]{Corollary}
\newtheorem{conj}[theorem]{Conjecture}

\theoremstyle{definition}
\newtheorem{definition}[theorem]{Definition}
\newtheorem{assumption}[theorem]{Assumption}
\newtheorem{remark}[theorem]{Remark}
\newtheorem{example}[theorem]{Example}

\DeclareMathOperator{\tor}{tor}

\DeclareMathOperator{\ext}{ext}

\begin{document}

\title{Betti Tables of Reducible Algebraic Curves}

\author[Bruce]{David J. Bruce}
\author[Kao]{Pin-Hung Kao}
\author[Nash]{Evan D. Nash}
\author[Perez]{Ben Perez}
\author[Vermeire]{Peter Vermeire}

\address{University of Michigan, Ann Arbor, MI, 48109}
\email{djbruce@umich.edu}

\address{Department of Mathematics, Central Michigan University, Mt Pleasant, MI, 48859}
\email{kao1p@cmich.edu}

\address{University of Nebraska-Lincoln, Lincoln, NE, 68588}
\email{E\_nash@cox.net}

\address{St. Olaf College,  Northfield, MN 55057}
\email{perez@stolaf.edu}

\address{Department of Mathematics, Central Michigan University, Mt Pleasant, MI, 48859}
\email{p.vermeire@cmich.edu}

\begin{abstract}
	We study the Betti tables of reducible algebraic curves with a focus on connected line arrangements and provide a general formula for computing the quadratic strand of the Betti table for line arrangements that satisfy certain hypotheses. We also give explicit formulas for the entries of the Betti tables for all curves of genus zero and one. Last, we give formulas for the graded Betti numbers for a class of curves of higher genus.
\end{abstract}
	
\thanks{The first, third, and fourth authors were supported by NSF grant DMS-1156890.}

\maketitle

	\section{Introduction}
	
	As in \cite{GraphCurves}, to a graph $G$ that satisfies Assumption~\ref{A:1.1} below, we associate a {\em graph curve} $\ol{G} \subset \P^n$ obtained by associating to each vertex of $G$ a line and to each edge connecting vertices a point of intersection of the corresponding lines. 
		
We know that a smooth curve of genus $g$ and degree $d \ge 2g+1+p$, $p \ge 0$, satisfies $N_{2,p}$~\cite{Green}. Thus the first $p+1$ entries in the quadratic strand of its Betti table are determined by its genus and degree. Line arrangements, on the other hand, are less uniformly behaved.

	Consider the following three graphs:
	
	\vspace{12pt}
	
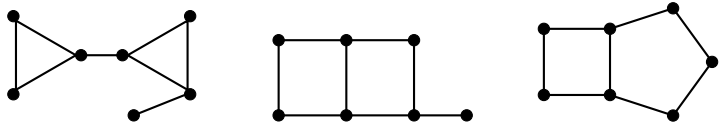
	\begin{center}
	\begin{tabular}{p{1.25in} p{1.25in} p{1.25in}}
	\begin{tikzpicture}
		\node [draw, rotate=90, thick, minimum size=1.25cm, regular polygon, regular polygon sides=3, scale=.85] at (1.25,0) {};
		\node [draw, rotate=-90, thick, minimum size=1.25cm, regular polygon, regular polygon sides=3, scale=.85] at (-.5,0) {};
		\draw[thick] (.1,0)--(.65,0);
		\draw[thick] (1.55,-.5)--(.8,-.8);
        \draw[fill] (.8,-.8) circle(.075);
        \foreach \x in {-30,90,210} \draw [fill,rotate=90,yshift=-1.25cm] (\x:0.6cm) circle(0.075);
        \foreach \x in {-30,90,210} \draw [fill,rotate=-90,yshift=-.5cm] (\x:0.6cm) circle(0.075);
	\end{tikzpicture}
	&
	\begin{tikzpicture}
		\draw[thick] (0,0)--(1.8,0)--(1.8,1)--(0,1)--cycle;
	    \draw[thick] (.9,0)--(.9,1);
	    \foreach \x in {0,.9,1.8} \draw[fill] (\x,0) circle(.075);
	    \foreach \x in {0,.9,1.8} \draw[fill] (\x,1) circle(.075);
	    \draw[thick] (1.8,0)--(2.5,0);
	    \draw[fill] (2.5,0) circle(.075);
	\end{tikzpicture}
	&
	\begin{tikzpicture}[scale=2]
		\node [draw, rotate=-90, thick, minimum size=1.5cm, regular polygon, regular polygon sides=5] at (0,0) {};
		\foreach \x in {18,90,90+72,90+2*72,90+3*72} \draw [fill,rotate=-90] (\x:0.375cm) circle(0.0375);
		\draw[fill] (-	.743cm,.22cm) circle(.0375);
		\draw[fill] (-	.743cm,-.22cm) circle(.0375);
		\draw[thick] (-.303cm,.22cm)--(-.743cm,.22cm)--(-.743cm,-.22cm)--(-.303cm,-.22cm);
	\end{tikzpicture}
	\end{tabular}
	\end{center} 
	
	\vspace{12pt}
	
These graphs all have associated graph curves of degree seven and genus two in $\P^5$. Using \emph{Macaulay 2} \cite{M2}, we find their corresponding Betti tables to be:

	\vspace{12pt}

	\begin{center}
	\begin{tabular}{c c c}
		\begin{tabular}{c | c c c c c}
			- & 0 &  1 &  2 &  3 & 4\\ \hline
			T & 1 & 10 & 20 & 15 & 4\\
			0 & 1 &  - &  - &  - & -\\
			1 & - &  8 & 14 &  9 & 2\\
			2 & - &  2 &  6 &  6 & 2
		\end{tabular}
		&
		\begin{tabular}{c | c c c c c}
			- & 0 & 1 &  2 &  3 & 4\\ \hline
			T & 1 & 8 & 15 & 11 & 3\\
			0 & 1 & - &  - &  - & -\\
			1 & - & 8 & 12 &  6 & 1\\
			2 & - & - &  3 &  5 & 2
		\end{tabular}
		&
		\begin{tabular}{c | c c c c c}
			- & 0 & 1 &  2 & 3 & 4\\ \hline
			T & 1 & 8 & 13 & 8 & 2\\
			0 & 1 & - &  - & - & -\\
			1 & - & 8 & 12 & 4 & -\\
			2 & - & - &  1 & 4 & 2
		\end{tabular}
	\end{tabular}
	\end{center}
	
	\vspace{6pt}
	
	We begin by setting some notations and assumptions that will be used throughout the paper. 	We work throughout over an algebraically closed field of characteristic zero. For a given graph $G$, we use $d$, $m$, and $f$ to represent the number of vertices, edges, and faces, respectively, of $G$. Also, for a given graph $G$, the number of vertices of degree $i$ is denoted by $x_i$. Lastly, we let $g = p_a(\ol{G})$ be the arithmetic genus of the curve $\ol{G}$ associated to the graph $G$.  Note that $g=m-d+1$.
	
\begin{assumption}\label{A:1.1}
	For the remainder of the paper, if $G$ is a graph, then $G$ satisfies the followings:
	
	\begin{enumerate}
		\item $G$ is planar.
		
		\item	$G$ is connected.
		
		\item	$G$ is simple.
		
		\item	$G$ is strictly subtrivalent.
		
		\item	Every connected subgraph of $G$ has $d \ge 2g+1$.
	\end{enumerate}
	We note that if $G$ satisfies these conditions, then $\ol{G} \subset \P^{d-g}$ is arithmetically Cohen-Macaulay (ACM) and non-special \cite{ACM}.
\end{assumption}

Though we do not use this in what is to come, it is a straightforward combinatorial exercise to show the following:

\begin{prop}\label{P:1.2}
	$G$ satisfies Assumption~\ref{A:1.1} if and only if no connected subgraph of $G$ has $x_2 < 3 $ and $x_3 = d - x_2$.\qed
\end{prop}

%\begin{prop}\label{P:1.3}
	%Let $G$ be a graph corresponding to a curve of genus $g$ that satisfies Assumption~\ref{A:1.1}, then $G$ is planar.
%\end{prop}

To fix notation, recall that a Betti table has the following form:

\vspace{6pt}

\begin{center}
\begin{tabular}{c | c c c c c}
	       - &        0 &         1 &        2 &        3 & $\cdots$\\ \hline
	       0 &        1 &   $\cdot$ &  $\cdot$ &  $\cdot$ & $\cdots$\\
	       1 &  $\cdot$ &  $b_{12}$ & $b_{23}$ & $b_{34}$ & $\cdots$\\
	       2 &  $\cdot$ &  $b_{13}$ & $b_{24}$ & $b_{35}$ & $\cdots$\\
	       3 &  $\cdot$ &  $b_{14}$ & $b_{25}$ & $b_{36}$ & $\cdots$\\
	$\vdots$ & $\vdots$ &  $\vdots$ & $\vdots$ & $\vdots$ & $\vdots$\\
\end{tabular}
\end{center}

\vspace{6pt}

\begin{definition}\label{D:1.4}
	Following \cite{EGHP}, we say that an embedded variety $X \subset \P^n$ satisfies $N_{d,p}$ if $b_{1,j} = 0$ for $j > d$ and $b_{i,d+i} = 0$ for $1 \le i \le p$.
\end{definition}

	We now give a general result from the literature and some immediate corollaries. See \cite{Green} and \cite{Laz} for an excellent introduction to these ideas.
	
\begin{prop}\label{P:1.5} \cite[5.8]{Eisenbud}
	Let $C \subset \P^n$ be an ACM curve with homogeneous coordinate ring $R$. Then there is an exact sequence
	\begin{align*}
	 0 &\xlongrightarrow{} \operatorname{Tor}_i(R,\mathbb{C})_k \xlongrightarrow{} H^1 (C,\wedge^{i+1}M_{L}(k-i-1)) \\
	   &\xlongrightarrow{}  H^1 (C,\wedge^{i+1}\Gamma (k-i-1)) \xlongrightarrow{}  H^1(C,\wedge^i M_L (k-i))  \xlongrightarrow{}  0.
	\end{align*}
	Where $\gG$ is trivial of rank $n+1$ and $M_L$ is the kernel of the surjection $\gG(C,\bo_C(1)) \to \bo_C(1)$.
\end{prop}

\begin{coro}\label{C:1.6}
	If $\ol{G} \subset \P^n$ is a curve of arithmetic genus $g$ associated to a graph $G$ satisfying Assumption~\ref{A:1.1}, then
	\begin{enumerate}
		\item	For $k \ge i+3$ we have $b_{i,k}(\ol{G}) = 0$, thus $\ol{G}$ is 3-regular;
		
		\item	$b_{i,i+2}(\ol{G}) = h^1(\ol{G}, \wedge^{i+1} M_L(1))$;
		
		\item	$b_{i,i+1}(\ol{G}) = h^1(\ol{G}, \wedge^{i+1} M_L) - g {n+1 \choose i+1} + b_{i-1,i+1}(\ol{G})$.
	\end{enumerate}
\end{coro}

\begin{proof}
	We use Proposition~\ref{P:1.5}.  First note that $b_{i,k}(\ol{G}) = \tor_i(R,\C)_k$ \cite{Laz}. As our graphs $G$ have $H^1(\ol{G},\bo(1)) = 0$, we have $H^1(\ol{G}, \wedge^{i+1} \gG(k-i-1)) = 0$ for $k \ge i+2$ proving both (1) and (2). (3) is straightforward.
\end{proof}

\begin{prop}\label{P:1.7}
	If $G$ satisfies Assumption~\ref{A:1.1}, then $H^0(\ol{G}, \wedge^k M_L) = 0$ for $k \ge 1$.
\end{prop}

\begin{proof}
	Consider the sequence
	\[
		0 \xrightarrow{} M_L \xrightarrow{} \gG(\ol{G},\bo(1)) \xrightarrow{} \bo(1) \xrightarrow{} 0.
	\]
	Since $\ol{G}$ is linearly normal, we have $H^0(\ol{G},M_L) = 0$. Tensoring the sequence with $M_L^{\otimes k-1}$ we see that $H^0(\ol{G},M_L^{\otimes k}) = 0$. As we work over a field of characteristic 0, $\wedge^k M_L$ is a summand of $M_L^{\otimes k}$ and the statement follows.
\end{proof}

Our main results are Theorems~\ref{P:2.12} and ~\ref{P:2.13} where we give explicit formulas for all curves of genus $0$ and $1$, and Theorem~\ref{P:3.6} where we give a formulas for a large class of curves of higher genus.  We end with a collection of conjectures.

\section{Curves of genus zero and one}

Corollary~\ref{C:1.6} enables us to find the closed formulas for Betti tables of various classes of graphs. In this section we provide such formulas for graphs of genus zero and genus one.
We first establish formulas for the Betti tables of paths and refer to \cite{Polygons} for the formulas for the Betti tables of cyclic graphs, and then show how to extend these to formulas for all graphs of genus zero and one. 

\subsection{Paths and Cyclic Graphs}

We first prove a general proposition that provides us some understanding of how to add a vertex of degree one to a graph $G$, which corresponds to adding a line $\l \in \P^n$ that intersects $\ol{G}$ transversally at a point.

\begin{prop}\label{P:2.1}
    Let $X$ and $Y$ be projective varieties in $\P^n$ and $\l \subset \P^n$ be a line such that $X = Y \cup \l$ and $\l$ intersects $Y$ transversally at a point. Then:
    \[
        H^1(X, \wedge^{i+1} M_L(1)) \iso H^1(Y, \wedge^{i+1}M_L(1)).
    \]
\end{prop}

\begin{proof}
    It is straightforward to check that $M_L|_\l = \bo(-1) \oplus T$, where $T$ is trivial of rank $n-1$. Thus
    \[
		\wedge^{i+1}M_L(1)|_\l = \wedge^{i+1}T(1) \oplus \wedge^i T
    \]
    and it follows that 
    \[
        H^0(\l,\wedge^{i+1}M_L(1)) \to H^0(Y \cap \l, \wedge^{i+1} M_L(1))
    \]
	is surjective. Hence $f$ in the larger sequence below is surjective:
	\begin{align*}
	    0 &\to H^0(Y,\wedge^{i+1}M_L(1))\\
	      &\to H^0(X,\wedge^{i+1}M_L(1)) \oplus H^0(\l,\wedge^{i+1}M_L(1)) \xrightarrow{f} H^0(\{p\}, \wedge^{i+1}M_L(1))\\
	      &\to H^1(Y, \wedge^{i+1}M_L(1)) \to H^1(X,\wedge^{i+1}M_L(1)) \oplus H^1(\l,\wedge^{i+1}M_L(1)) \to 0
	\end{align*}
	Further, the above argument also implies $H^1(\l,\wedge^{i+1} M_L(1)) = 0$ and therefore $H^1(X, \wedge^{i+1} M_L(1)) \iso H^1(Y, \wedge^{i+1}M_L(1))$.
\end{proof}

We give a formula for the graded Betti numbers for the paths.

\begin{theorem}\label{T:2.2}
	Let $P_n$ be the path on $n$ vertices, and $\ol{P}_n$ be the corresponding curve in $\P^n$. Then for $i \ge 1$,
	\begin{align*}
		b_{i,i+1}(\ol{P}_n) &= n{n-1 \choose i} - {n \choose i+1},\\
		b_{i,i+2}(\ol{P}_n) &= 0.
	\end{align*}	
\end{theorem}

\begin{proof}
	We need to compute $h^1(\ol{P}_n, \wedge^{i+1}M_L(k-i-1))$. When $k=i+1$, we look at the following exact sequence
	\begin{align*}
		0 &\longrightarrow H^0(\ol{P}_n, \wedge^{i+1} M_L) \longrightarrow \bigoplus_{j=1}^n H^0(\l_j, \wedge^{i+1} M_L) \longrightarrow H^0(A, \wedge^{i+1} M_L)\\
		  &\longrightarrow H^1(\ol{P}_n,\wedge^{i+1} M_L) \longrightarrow \bigoplus_{j=1}^n H^1(\l_n,\wedge^{i+1}M_L) \longrightarrow 0,
	\end{align*}
	where $\l_1, \ldots, \l_n$ are the lines that comprise $\ol{P}_n$ and $A = \set{p_1, \ldots, p_{n-1}}$ is the set of points where these lines intersect. By Proposition~\ref{P:1.7} and the fact that $H^1(\l_i,\wedge^{i+1}M_L)=0$ for all $i$, the above becomes
	\begin{align*}
		0 &\longrightarrow \bigoplus_{j=1}^n H^0(\l_j,\wedge^{i+1}M_L) \longrightarrow H^0(A,\wedge^{i+1}M_L)\\
		  &\longrightarrow H^1(\ol{P}_n, \wedge^{i+1}M_L) \longrightarrow 0.
	\end{align*}
	Thus
	\[
		h^1(\ol{P}_n, \wedge^{i+1}M_L) = h^0(A,\wedge^{i+1}M_L) - \sum_{j=1}^n h^0(\l_i,\wedge^{i+1}M_L).
	\]
	Since $\wedge^{i+1} M_L$ has dimension ${n \choose i+1}$ and there are $n-1$ points, it follows that
	\[
		h^0(A,\wedge^{i+1} M_L) = (n-1){n \choose i+1}.
	\]
	We now note that $M_L|_{\l_i}$ has rank $n-1$, thus
	\[
		h^0(\l_j,\wedge^{i+1}M_L) = {n-1 \choose i+1}.
	\]
	
	Next, by Proposition~\ref{P:2.1} we have
	\[
		H^1(\ol{P}_n, \wedge^{i+1}M_L(1)) \iso \bigoplus_{j=1}^n H^1(\l_j,\wedge^{i+1} M_L(1)) = 0.
	\]
	Therefore $b_{i,i+2}(\ol{P}_n) = 0$ for all $i \ge 1$. Lastly, paths are of genus zero, so by part (3) of Corollary~\ref{C:1.6}, we have 
	\[ 
		b_{i,i+1}(\ol{P}_n) = (n-1){n \choose i+1} - n{n-1 \choose i+1} = n{n-1 \choose i}-{n\choose i+1}.
	\]
	\end{proof}

	A formula for the graded Betti numbers for cyclic graphs can be obtained immediately from \cite{Polygons}:
	
	\begin{theorem}\label{T:2.3}
		Let $C_{n+1}$ be the cyclic graph on $n+1$ vertices, and $\ol{C}_{n+1}$ be the corresponding curve in $\P^n$. Then for all $i < n$,
		\[
			b_{i,i+1}(\ol{C}_{n+1}) = n{n-1 \choose i} - {n-1 \choose i-1} - {n \choose i+1}; b_{i-1,i+1}(\ol{C}_{n+1}) = 0
		\]		
		and for $i=n$
		\[
			b_{n,n+1}(\ol{C}_{n+1})=0; b_{n-1,n+1}(\ol{C}_{n+1}) = 1.
		\]\qed
	\end{theorem}
	
\subsection{$\boldsymbol{d}$-Extensions}

	\begin{definition}\label{D:2.4}
		Let $G$ and $H$ be graphs.  We say $H$ is a \emph{$d$-extension} of $G$ if there exists a sequence of strictly subtrivalent graphs $G_{0},\ldots,G_{d}$  such that $G=G_{0}$, $H=G_{d}$ and $G_{t}$ can be obtained from $G_{t-1}$ by adding a degree one vertex and an edge.
	\end{definition}
	
	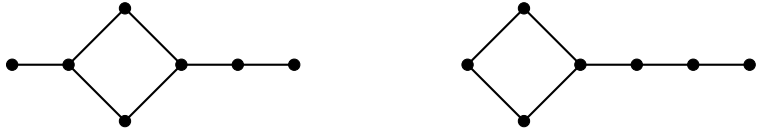
\begin{figure}[h]
	\begin{center}
	\begin{tikzpicture}[scale=.75]
		\draw[thick] (-1,0)--(0,-1)--(1,0)--(0,1)--cycle;
		\draw[thick] (-2,0)--(-1,0);
		\draw[thick] (1,0)--(3,0);
		\draw[fill] (-2,0) circle(.1);
		\draw[fill] (-1,0) circle(.1);
		\draw[fill] (1,0) circle(.1);
		\draw[fill] (2,0) circle(.1);
		\draw[fill] (3,0) circle(.1);
		\draw[fill] (0,1) circle(.1);
		\draw[fill] (0,-1) circle(.1);
	\end{tikzpicture}
	\hspace{.75in}
	\begin{tikzpicture}[scale=.75]
		\draw[thick] (-1,0)--(0,-1)--(1,0)--(0,1)--cycle;
		\draw[thick] (1,0)--(4,0);
		\draw[fill] (0,-1) circle(.1);
		\draw[fill] (0,1) circle(.1);
		\draw[fill] (-1,0) circle(.1);
		\draw[fill] (1,0) circle(.1);
		\draw[fill] (2,0) circle(.1);
		\draw[fill] (3,0) circle(.1);
		\draw[fill] (4,0) circle(.1);
	\end{tikzpicture}
	\caption{3-Extensions of $C_{4}$}
	\end{center}
	\end{figure}
	
	Computing the Betti tables for $d$-extensions involves taking the variety corresponding to a graph embedded in $\P^r$ and putting it in $\P^{r+d}$. 
	
	\begin{prop}\label{P:2.5}
		Suppose that $X \subseteq \P^r \subset \P^n$ is a projective variety which spans $\P^r$. Let $M_L = \Omega_{\P^n}(1) \otimes \bo_X$ and $\widetilde{M}_L = \Omega_{\P^r}(1) \otimes \bo_X$, then
		\[
			h^j(X,\wedge^{i+1}M_L(k-i-1)) = \sum_{t=0}^{n-r} {n-r \choose t} h^j(X,\wedge^{i+1-t} \widetilde{M}_L(k-i-1)).
		\]
	\end{prop}
	
	\begin{proof}
		Consider a linear embedding $\P^{n-1} \subset \P^n$. We have the conormal-cotangent sequence
		\[
			0 \to \bo_{\P^{n-1}}(-1) \to \Omega_{\P^n} \otimes \bo_{\P^{n-1}} \to \Omega_{\P^{n-1}} \to 0.
		\]
		As
		\[
			\ext^1(\Omega_{\P^{n-1}}, \bo_{\P^{n-1}}(-1)) = h^{n-2}(\P^{n-1},\Omega_{\P^{n-1}}(-n-1)) = 0
		\]
		the sequence splits. Similarly, given a linear embedding $\P^r \subset \P^n$ and noting that $M_L = \Omega_{\P^n}(1)$ we have $M_L \otimes \bo_{\P^r} = \widetilde{M}_L \oplus T$ where $T$ is trivial of rank $n-r$. Thus
		\[
			\wedge^{i+1} M_L = \bigoplus_{t=0}^{n-r} \wedge^{i+1-t} \widetilde{M}_L \otimes \wedge^t T.
		\]
		As $T$ is trivial, the result follows.
	\end{proof}

	\begin{prop}\label{T:2.6}
		Let $X \subseteq \P^n$ and $Y = X \cup \l \subset \P^{n+1}$ be projective varieties such that $\l$ intersects $X$ transversally at a single point. If $i \ge 1$, then
		\[
			b_{i,i+2}(Y) = b_{i,i+2}(X) + b_{i-1,i+1}(X).
		\]
	\end{prop}
	
	\begin{proof}
		By Proposition~\ref{P:2.1} and Proposition~\ref{P:2.5}, it follows that
		\begin{align*}
			h^1(Y, \wedge^{i+1}M_L(1)) 
				&= h^1(X,\wedge^{i+1}M_L(1))\\
				&= \sum_{t=0}^1 {1 \choose t} h^1(X,\wedge^{i+1-t}M_L(1))\\
				&= \sum_{t=0}^1 {1 \choose t} b_{i-t,k-t}(X)\\
				&= b_{i,i+2}(X) + b_{i-1,i+1}(X).\qedhere
		\end{align*}
	\end{proof}

	The above theorem tells us that the Betti table of any $d$-extension can be written solely in terms of the Betti table of the original graph. A particularly nice instance of Theorem~\ref{T:2.3} is when $G$ is an $n$-cycle:
	
	\begin{coro}\label{C:2.7}
		Let $G$ be any $d$-extension of an $n$-cycle. Then
		\[
			b_{n-1+j,n+j}(G) = {d \choose j}.
		\]
	\end{coro}
	
	\begin{proof}
		Let $G$ be an $d$-extension of an $n$-cycle. Then by Proposition~\ref{T:2.6}
		\[
			b_{n-1+j,n+j} = \sum_{t=0}^d {d \choose t} b_{n+j-1-t,n+j-t}.
		\]
		Since the cubic strand of an $n$-cycle is zero everywhere except for the $(n-1)$th column where it has a $1$, it follows that the above sum has a nonzero term only when $t = j$. Thus $b_{n-1+j,n+j}(G) = {d \choose j}$.
	\end{proof}
	
	\begin{coro}
	    Let $G$ be a graph satisfying Assumption~\ref{A:1.1} with corresponding curve $\ol{G} \subset \P^r$. If $H$ is a $d$-extension of $G$ then $\ol{H} \subset \P^{r+d}$ satisfies $N_{2,p}$ if and only if $\ol{G} \subset \P^r$ satisfies $N_{2,p}$.
	\end{coro}

	\begin{proof}
	    It follows from Proposition~\ref{T:2.6} that
	    \[
	        b_{i,i+2}(\ol{H}) = \sum_{t=0}^d {d \choose t} b_{i-t,i+2-t}(\ol{G}).
	    \]
		Since $\ol{G}$ satisfies $N_{2,p}$, we know that for $1 \le i \le p$, $b_{i,i+2}(\ol{G}) = 0$, and thus $\ol{H}$ also satisfies $N_{2,p}$.
	\end{proof}

\subsection{The Quadratic Strand}

\begin{theorem}\label{T:2.9}
    If $\ol{G} \subset \P^n$ is a graph curve of genus $g$ associated to a graph $G$ that satisfies Assumption~\ref{A:1.1}, then for $i \ge 1$,
    \[
        b_{i,i+1}(\ol{G}) = n {n-1 \choose i} - g {n-1 \choose i-1} - {n \choose i+1} + b_{i-1,i+1}(\ol{G}).
    \]
\end{theorem}

\begin{proof}
    We have, from Corollary~\ref{C:1.6}, that
    \begin{equation}\label{E:2.1}
    		b_{i,i+1}(\ol{G}) = h^1(\ol{G}, \wedge^{i+1}M_L) - g{n+1 \choose i+1} + b_{i-1,i+1}(\ol{G}).
    \end{equation}
    Moreover, $h^0(\ol{G}, \wedge^{i+1}M_L) = 0$ by Proposition~\ref{P:1.7}. The Euler characteristic $\gx(\ol{G}, \wedge^{i+1}M_L)$ of $\ol{G}$ gives us    \begin{align*}
    		\gx(\ol{G},\wedge^{i+1}M_L)
			&= h^0(\ol{G},\wedge^{i+1}M_L) - h^1(\ol{G},\wedge^{i+1}M_L)\\
			&= - h^1(\ol{G},\wedge^{i+1}M_L).
    \end{align*}
	Let $\ol{G} = \bigcup_{i=1}^d \l_i$, where $\l_i$ are lines in $\P^n$ and the $p_i$ their points of intersections. As in the proof of Theorem~\ref{T:2.2} we have
	\begin{align*}
		0 &\to H^0(\ol{G}, \wedge^{i+1}M_L) \to \bigoplus_{i=1}^d H^0(\l_i,\wedge^{i+1}M_L|_{\l_i})\\
		  &\to H^0(\set{p_1, \ldots, p_m}, \wedge^{i+1}M_L) \to 0.
	\end{align*}
	Since $m = d + g - 1$ and $n = d - g$, the above imply that
	\begin{align*}
		\gx(\ol{G}, \wedge^{i+1}M_L)
			&= \bigoplus_{i=1}^d \gx(\l_i,\wedge^{i+1}M_L|_{\l_i}) - \gx(\set{p_1, \ldots, p_m}, \wedge^{i+1}M_L)\\
			&= d{n-1 \choose i+1} - m{n \choose i+1}\\
			&= (g+n){n-1 \choose i+1} - (2g+n-1){n \choose i+1}.
	\end{align*}
	Thus $h^1(\ol{G},\wedge^{i+1}M_L) = (2g+n-1){n \choose i+1} - (g+n){n-1 \choose i+1}$. Substitute this into~\eqref{E:2.1} to obtain the desired result.
\end{proof}

The followings is known for smooth curves \cite[8.6]{Eisenbud}:

\begin{coro}\label{C:2.10}
	If $\ol{G} \subset \P^n$ is a graph curve of genus $g$ associated to a planar graph $G$ satisfying Assumption~\ref{A:1.1}, then
	\[
		b_{n-1,n+1}(\ol{G}) = g.
	\]
\end{coro}

\begin{proof}
	By \cite{ACM}, $\ol{G}$ is ACM, so by the Auslander-Buchsbaum formula, the projective dimension of $\ol{G}$ is equal to $n-1$. Thus by Proposition~\ref{T:2.6}
	\[
		0 = b_{n,n+1}(\ol{G}) = n{n-1 \choose n} - g{n-1 \choose n-1} - {n \choose n+1} + b_{n-1,n+1}(\ol{G}).
	\]
	That is, $b_{n-1,n+1}(\ol{G}) = g$.
\end{proof}

\subsection{Genus Zero and One}

All graphs of genus zero are trees and all graphs of genus one are $d$-extensions of cyclic graphs. Thus with the results from the previous subsections, we now give formulas for the Betti tables of all curves of genus zero and genus one.

\begin{theorem}\label{P:2.12}
	If $\ol{G} \subset \P^n$ is a graph curve of genus zero, then for $i \ge 1$,
	\begin{align*}
		b_{i,i+1}(\ol{G}) &= n{n-1 \choose i} - {n \choose i+1},\\
		b_{i,i+2}(\ol{G}) &= 0.
	\end{align*}
\end{theorem}

\begin{proof}
	It follows from Theorems~\ref{T:2.2} and~\ref{T:2.9} that the quadratic strand of the Betti table for any line arrangement depends only on the degree and genus of the curve. Therefore all curves of genus zero and degree $d$  have the same quadratic strand as the curve represented by the path on $d$ vertices. Similarly, we can conclude that all $d$-extensions have the same cubic strand. The result then follows from Theorem~\ref{T:2.2}.
\end{proof}

\begin{theorem}\label{P:2.13}
	If $G$ is a $d$-extension of $C_{r+1}$ with corresponding curve $\ol{G} \subset \P^n$ of genus one, then for $i \ge 1$,
	\begin{align*}
		b_{i,i+1}(\ol{G}) &= n{n-1 \choose i} - {n-1 \choose i-1} - {n \choose i+1} + {d \choose i-r},\\
		b_{i,i+2}(\ol{G}) &= {d \choose i-r+1}.
	\end{align*}
\end{theorem}

\begin{proof}
	$G$ is a $d$-extension of $C_{r+1}$, so by Proposition~\ref{T:2.6} we have
	\[
		b_{i,i+2}(\ol{G}) = \sum_{t=0}^d {d \choose t} b_{i-t,i+2-t}(\ol{C}_{r+1}).
	\]
	Furthermore, by Theorem~\ref{T:2.3}, we know that $b_{i-t,i+2-t}(\ol{C}_{r+1})$ vanishes except when $i-t = r-1$. So the above simples to 
	\[
		b_{i,i+2}(\ol{G}) = {d \choose i-r+1}.
	\]
	Applying Theorem~\ref{T:2.9}, we obtain the desired result.
\end{proof}

\section{Curves of Higher Genus}

Theorem~\ref{T:2.9} enables us to compute the quadratic strand of any curve, regardless of its genus, in terms of its degree, genus, and the cubic strand. We apply this to find the graded Betti numbers for curves of higher genus. Although it is difficult to calculate the cubic strand in all cases, we are able to calculate the cubic strand for a class of higher genus graphs.

\begin{prop}\label{T:3.1}
    Let $\ol{G} = \ol{G}_1 \cup \ol{G}_2 \subset \P^n$ be a non-degenerate curve such that $\ol{G}_1 \cap \ol{G}_2 = \set{p}$ is a reduced point. If $\ol{G}_i$ spans $\P^{n_i}$, then
    \[
        b_{i,i+2}(\ol{G}) = \sum_{s=0}^{n-n_1} {n-n_1 \choose s} b_{i-s,i-s+2}(\ol{G}_1) + \sum_{t=0}^{n-n_2} {n-n_2 \choose t} b_{i-t,i-t+2}(\ol{G}).
    \]
\end{prop}

\begin{remark}
    We note that Proposition~\ref{P:2.1} is a special case of the above theorem with $\ol{G}_2 = \l \subset \P^n$ being a line that intersects $\ol{G}_1$ transversally at a point.
\end{remark}

\begin{proof}[Proof of Proposition~\ref{T:3.1}]
    Consider the sequence
    \begin{align*}
        0 &\to H^0(\ol{G},\wedge^{i+1}M_L(1)) \to H^0(\ol{G}_1,\wedge^{i+1}M_L(1)) \oplus H^0(\ol{G}_2,\wedge^{i+1}M_L(1))\\
          &\xrightarrow{f} H^0(\set{p},\wedge^{i+1}M_L(1)) \to H^1(\ol{G},\wedge^{i+1}M_L(1))\\
          &\to H^1(\ol{G}_1,\wedge^{i+1}M_L(1)) \oplus H^1(\ol{G}_2,\wedge^{i+1}M_L(1)) \to 0.
    \end{align*}
	The surjection $\wedge^{j+1} \gG \to \wedge^j M_L(1)$ implies that $\wedge^j M_L(1)$ is globally generated on $\ol{G}_i$ for $j \ge 1$; this immediately implies $f$ is surjective and the result follows by Proposition~\ref{P:2.5}.
\end{proof}

\begin{coro}\label{C:3.3}
    With hypotheses and notation as in Proposition~\ref{T:3.1}, $\ol{G}_1$ and $\ol{G}_2$ satisfy $N_{2,p}$ if and only if $\ol{G}$ does also.
\end{coro}

We now show, via an example, how to compute the cubic strand of a higher genus curve by applying Proposition~\ref{T:3.1}.

\begin{example}\label{E:3.4}
    Let $G$ be a graph of arithmetic genus 3 shown below.
    
    \vspace{6pt}
    
    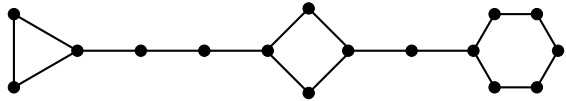
\begin{figure}[h]
    \begin{center}
    \begin{tikzpicture}[scale=.75]
		\node [draw, rotate=-90, thick, minimum size=1.125cm, regular polygon, regular polygon sides=3] at (0,0) {};
		\foreach \x in {-30,90,210} \draw [fill, rotate=-90] (\x:0.75cm) circle(.1);
		\draw[thick, rotate=-90] (0,.75cm)--(0,1.875cm)--(0,3cm)--(0,4.125cm);
		\foreach \x in {1.875,3,4.125} \draw[fill, rotate=-90] (0,\x cm) circle(.1);
		\node [draw, thick, minimum size=1.125cm, regular polygon, regular polygon sides=4,rotate=45] at (4.85cm,0) {};
		\draw[fill] (5.55cm,0) circle(.1);
		\draw[fill] (4.85cm,.75cm) circle(.1);
		\draw[fill] (4.85cm,-.75cm) circle(.1);
		\draw[thick] (5.55cm,0)--(6.675cm,0)--(7.8cm,0);
		\draw[fill] (6.675cm,0) circle(.1);
		\node [draw, thick, minimum size=1.125cm, regular polygon, regular polygon sides=6] at (8.52cm,0) {};
		\foreach \x in {0,60,...,300} \draw [fill,xshift=8.52cm] (\x:.75cm) circle(0.1);
	\end{tikzpicture}
	\caption{Graph of arithmetic genus 3.}
	\end{center}
	\end{figure}
	
	$\ol{G} \subset \P^{12}$ is a curve of degree 16 and genus 3. Applying Proposition~\ref{T:3.1} twice we see that the formula for the cubic strand is
	\begin{align*}
	    b_{i,i+2}(\ol{G})
			&= \sum_{r=0}^{11} {11 \choose r} b_{i-r,i-r+2}(\ol{C}_3) + \sum_{s=0}^{10} {10 \choose s} b_{i-s,i-s+2} (\ol{C}_4)\\
			&\phantom{\sum_{r=0}^{11}} + \sum_{t=0}^8 {8 \choose t} b_{i-t,i-t+2}(\ol{C}_6)\\
			&= {11 \choose i-1} + {10 \choose i-2} + {8 \choose i-4}.
	\end{align*}
	Now apply Theorem~\ref{T:2.9} to obtain the Betti table of $G$:
	
	\begin{center}
	\begin{table}[h]\scalebox{.85}
	{\begin{tabular}{c | c c c c c c c c c c c c c}
		- & 0 &  1 &   2 &    3 &    4 &    5 &   6  &    7 &    8 &    9 &  10 &  11 & 12 \\ \hline
	    T & 1 & 76 & 549 & 2024 & 4764 & 7764 & 9078 & 7707 & 4724 & 2040 & 589 & 102 &  8 \\
	    0 & 1 &  - &   - &    - &    - &    - &    - &    - &    - &    - &   - &   - &  - \\
	    1 & - & 75 & 537 & 1959 & 4553 & 7306 & 8378 & 6937 & 4114 & 1699 & 461 &  73 &  5 \\
	    2 & - &  1 &  12 &   65 &  211 &  458 &  700 &  770 &  610 &  341 & 128 &  29 &  3 \\
	\end{tabular} }
	\vspace{6pt}
	\caption{Betti table for Example \ref{E:3.4}.}
	\end{table}
	\end{center}
\end{example}

The above example shows that by applying Proposition~\ref{T:3.1} and~\ref{T:2.9}, we are able to fully describe the Betti table of any \emph{tree of cycles}.

\begin{definition}\label{D:3.5}
    A graph $G$ is a \emph{tree of cycles} if $G$ can be obtained from a tree by replacing a finite number of non-adjacent edges with cyclic graphs.
\end{definition}

We generalize Example~\ref{E:3.4} with the following proposition, obtained by direct computation.

\begin{theorem}\label{P:3.6}
    Suppose $G$ is a tree of cycles satisfying Assumption~\ref{A:1.1} that comprised of $k_j$ cycles of length $j$. Then
    \begin{align*}
        b_{i,i+1}(\ol{G}) &= n{n-1 \choose i} - \left( \sum_{j=3}^\oo k_j \right){n-1 \choose i-1} - {n \choose i+1}\\
                          &\phantom{n{n-1 \choose i}}+ \sum_{j=3}^\oo k_j {n-j+1 \choose i-j+1}\\
        b_{i,i+2}(\ol{G}) &= \sum_{j=3}^\oo k_j {n-j+1 \choose i-j+2}.
    \end{align*}
\end{theorem}

Proposition~\ref{P:3.6} shows how the geometric properties of a graph can directly manifest themselves in the Betti table of the corresponding graph curves. We summarize them in the following two corollaries, the first is a specific case of a conjecture proposed by Burnham, Rosen, Sidman, and Vermeire in \cite{Burnham12}.

\begin{coro}\label{C:3.7}
    If $G$ is a tree of cycles with girth $\gr$ with corresponding graph curve $\ol{G} \subset \P^n$, then $b_{\gr-2,\gr}(\ol{G})$ is equal to the number of cycles of length $\gr$ in $G$.
\end{coro}

\begin{proof}
    It follows from Proposition~\ref{P:3.6} that
    \[
        b_{\gr-2,\gr}(\ol{G}) = \sum_{j=3}^\oo k_j {n-j+1 \choose \gr-j}.
    \]
	Since $\gr$ is the girth of $G$, $k_j = 0$ for all $j < \gr$. Moreover, $\gr - j < 0$ when $j > \gr$, so the above simplifies to
	\[
	    b_{\gr-2,\gr}(\ol{G}) = \sum_{j=3}^\oo k_j {n-j+1 \choose \gr-j} = k_\gr {n-\gr+1 \choose \gr-\gr} = k_\gr.\qedhere
	\]
\end{proof}

\begin{coro}\label{C:3.8}
    If $G$ is a tree of cycles with corresponding graph curve $\ol{G} \subset \P^n$, then $b_{n-1,n}(\ol{G})$ is equal to the number of bridges in $G$.
\end{coro}

\begin{proof}
    By Proposition~\ref{P:3.6}, we have
    \[
        b_{n-1,n}(\ol{G}) = 2 \left( \sum_{j=3}^\oo k_j \right) - \sum_{j=3}^\oo jk_j + n - 1.
    \]
	Since the genus of $G$ equals $\sum_{j=3}^\oo k_j$ and since $m = d+g-1$ and $n=2d-m-1$, the above simplifies to
	\[
	    b_{n-1,n}(\ol{G}) = m - \sum_{j=3}^\oo jk_j.
	\]
	This equals the number of edges of $G$ which are not part of cycles. In the case of a tree of cycles, this is exactly the number of bridges in $G$.
\end{proof}

\section{Conjectures and Future Work}

	\emph{Macaulay 2} has been an invaluable tool for our research. A large number of Betti tables were computed using it and these examples guided many of the results in this paper. It also motivated some conjectures that are mentioned in this section.
	
	In \cite{Burnham12} the authors conjectured that various structural properties of a graph could be used to calculate entries of the Betti table for the corresponding curve.
	
	\begin{conj}\label{Cj:4.1}
	    Let $G$ be a graph with $d$ vertices and girth $\gamma$. If $d = 2g + 1 + p$ and $\gamma-2 \le p$ then $b_{\gamma-2,\gamma}(\ol{G})$ is equal to the number of $\gamma$-cycles in $G$. 
	\end{conj}
	
	All of our examples have thus far supported this conjecture. Moreover, Propositions~\ref{P:2.12} and~\ref{P:2.13} prove the above conjecture for graphs of genus zero and one respectively. (See also Proposition~\ref{C:3.7}.)
	
	We also noticed, in our database of Betti tables, similar examples of structural properties of the graph corresponding directly to entries of the Betti table. Corollary~\ref{C:2.10} is such an example; showing the last non-zero entry in the cubic strand is the genus of the graph.
	
	There also seems to be a similar pattern regarding the number of bridges in a graph. For example, consider the graphs below:
	
	\begin{center}
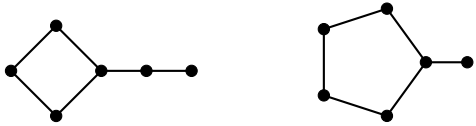
	\begin{table}[h]
	\begin{tabular}{p{1.5in} p{1.5in}}
	    \begin{tikzpicture}[scale=.6]
	          \draw[thick] (-1,0)--(0,-1)--(1,0)--(0,1)--cycle;
	          \draw[thick] (1,0)--(3,0);
	          \draw[fill] (-1,0) circle(.125);
			  \draw[fill] (1,0) circle(.125);
			  \draw[fill] (0,1) circle(.125);
			  \draw[fill] (0,-1) circle(.125);
			  \draw[fill] (2,0) circle(.125);
			  \draw[fill] (3,0) circle(.125);	        
	    \end{tikzpicture}
		&
		\begin{tikzpicture}[scale=2]
			  \node [draw, rotate=-90, thick, minimum size=1.5cm, regular polygon, regular polygon sides=5] at (0,0) {};
			  \foreach \x in {18,90,90+72,90+2*72,90+3*72} \draw [fill,rotate=-90] (\x:0.375cm) circle(0.0375);
			  \draw[thick] (0.375cm,0)--(.65cm,0);
			  \draw[fill] (.65cm,0) circle(.0375);		    
		\end{tikzpicture}
	\end{tabular}
	\vspace{6pt}
	%\caption{Graphs representing two curves in $\P^5$.}
	\end{table}
	\end{center}
	
	\vspace{6pt}

	Both of these graphs represent genus one curves in $\P^5$. However, the graph on the left has two bridges---edges which if removed result in a disconnected graph---while the graph on the right has only one bridge. Examining the Betti tables for the curves corresponding to these graphs indicates that some of the entries are the same while the others are different. In particular, $b_{4,5} = 2$ for the graph on the left and $b_{4,5} = 1$ for the one on the right.

	\begin{table}[h!]
	\begin{center}
	\begin{tabular}{ c | c  c  c cc} 
		- & 0 & 1 & 2 &3&4\\ \hline 
		T & 1 & 9  & 17 & 12&3 \\  
		0 & 1 & - & - & - &-\\  
		1 & - & 9 & 16 &10&2 \\ 
		2 & - & - & 1&2&1 \\  
	\end{tabular} 
	\quad \quad 
	\begin{tabular}{c|ccccc}
		-&0&1&2&3&4\\ \hline 
		T&1&9&16&10&2 \\ 
		0&1&-&-&-&-\\  
		1&-&9&16&9&1\\  
		2&-&-&-&1&1\\  
	\end{tabular}
	\end{center}
	
	\caption{Betti tables for curves represented by the graphs in Table 2.}
	\end{table}

	\begin{conj}\label{Cj:4.2}
		Let $\overline{G}\subset \P^{n}$ be a graph curve. Then $b_{n-1,n}(\overline{G})$ is equal to the number of bridges in $G$.
	\end{conj}

	Although we are as of yet unable to prove this in general we are capable of showing this true for various classes of graphs.  Namely, by applying Proposition~\ref{P:2.13} we can easily see that if $\overline{G}\subset \P^{n}$ is a genus zero graph curve then $b_{n-1,n}(\overline{G})=n-1$, as expected since all edges in a genus zero curve are bridges.  Similarly, since cyclic graphs contain no bridges, Corollary~\ref{C:2.7} also shows this conjecture to be true for cyclic graphs. Futher, Corollary~\ref{C:3.8} shows that the conjecture holds for trees of cycles.  Note also that by applying Proposition~\ref{T:3.1}, it is enough to verify this in the case where $G$ has no bridges.

	Looking through our examples it appears that the Betti tables of various classes of graphs of higher genera follow some nice patterns.  For example, consider the following graphs formed by gluing 4-cycles together along one edge. Let $C_{n}^{k}$ denote the graph made by gluing $k$ copies of $C_{n}$ together along single edges as done below.
	
	\begin{figure}[h!]
	\begin{center} 
	\begin{tikzpicture}[scale=.7,colorstyle/.style={circle, draw=black!100,fill=black!100, thick, inner sep=0pt, minimum size=2 mm}]
		\node at (-9,1)[colorstyle]{};
		\node at (-7,1)[colorstyle]{};
		\node at (-5,1)[colorstyle]{};
		\node at (-3,1)[colorstyle]{};
		\node at (-9,-1)[colorstyle]{};
		\node at (-7,-1)[colorstyle]{};
		\node at (-5,-1)[colorstyle]{};
		\node at (-3,-1)[colorstyle]{};
		\draw[thick](-9,1)--(-7,1)--(-5,1)--(-3,1);
		\draw[thick](-9,-1)--(-7,-1)--(-5,-1)--(-3,-1);
		\draw[thick](-9,-1)--(-9,1);
		\draw[thick](-7,-1)--(-7,1);
		\draw[thick](-5,-1)--(-5,1);
		\draw[thick](-3,-1)--(-3,1);
	\end{tikzpicture}

	\begin{tikzpicture}[scale=.7,colorstyle/.style={circle, draw=black!100,fill=black!100, thick,
inner sep=0pt, minimum size=2 mm}]
		\node at (-1,1)[colorstyle]{};
		\node at (1,1)[colorstyle]{};
		\node at (3,1)[colorstyle]{};
		\node at (5,1)[colorstyle]{};
		\node at (-1,-1)[colorstyle]{};
		\node at (1,-1)[colorstyle]{};
		\node at (3,-1)[colorstyle]{};
		\node at (5,-1)[colorstyle]{};
		\node at (7,-1)[colorstyle]{};
		\node at (7,1)[colorstyle]{};

		\draw[thick](-1,1)--(1,1)--(3,1)--(5,1)--(7,1);
		\draw[thick](-1,-1)--(1,-1)--(3,-1)--(5,-1)--(7,-1);
		\draw[thick](-1,-1)--(-1,1);
		\draw[thick](1,-1)--(1,1);
		\draw[thick](3,-1)--(3,1);
		\draw[thick](5,-1)--(5,1);
		\draw[thick](7,-1)--(7,1);
	\end{tikzpicture}
	\end{center}
	\caption{Top: $C_{4}^{3}$ representing a genus three curve in $\P^{5}$.  Bottom: $C_{4}^{4}$ representing a genus four curve in $\P^{6}$.} 
	\end{figure}
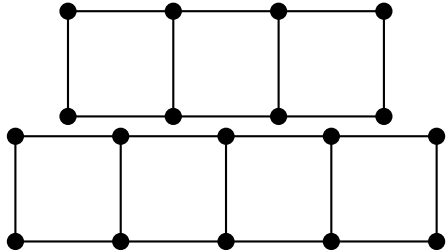

	\begin{table}[h!]
	\begin{tabular}{ c | c c c c c} 
		- & 0 & 1 &  2 &  3 & 4\\ \hline 
		T & 1 & 7 & 14 & 11 & 3\\  
		0 & 1 & - &  - &  - & -\\  
		1 & - & 7 &  8 &  3 & -\\  
		2 & - & - &  6 &  8 & 3\\  
	\end{tabular}
	\quad \quad 
	\begin{tabular}{ c | c c c c c c}
	- & 0 & 1 & 2 &3&4&5\\ \hline 
	T & 1 & 11 & 30 & 35 &19&4 \\ 
	0 & 1 & - & - & - &-&-\\ 
	1 & - & 11 & 20&15&4&- \\  
	2&- & - & 10 &20&15&4  \\   
	\end{tabular}
	\caption{Betti tables for $C_{4}^{3}$ and $C_{4}^{4}$ respectively.}
	\end{table}

	Looking at the Betti tables for these graphs one is struck by the apparent symmetry between the quadratic and cubic strands,  Upon further inspection of these and the Betti tables for other $C_{4}^{k}$ we have noticed quite a few other patterns, which characterize a good portion of these Betti tables.
	\begin{conj}
		Let $G$ be a graph comprised of $k$ $C_{4}$'s glued together along one edge. Then:
		\begin{align*}
		      b_{2,4}(\overline{G}) &= b_{1,2}(\overline{G})-1,\\
		    b_{i,i+1}(\overline{G}) &= b_{i+2,i+4}(\overline{G}), \quad i\geq 2,\\
		    b_{k,k+2}(\overline{G}) &=\left(b_{k+1,k+3}(\overline{G})\right)^{2}-1.
		\end{align*}
	\end{conj}
	
\bibliographystyle{plain}
\bibliography{ref}
\end{document}